\documentclass[11pt]{amsart}
\usepackage{hyperref,enumerate}
\usepackage{amssymb,amsmath,amsrefs}
\usepackage{MnSymbol}

\newcommand{\A}{\mathbb{A}}

\newcommand{\Q}{\mathbb{Q}}
\newcommand{\Z}{\mathbb{Z}}

\newcommand{\PP}{\mathbb{P}}

\newcommand{\calF}{\mathcal{F}}

\newcommand{\calP}{\mathcal{P}}
\newcommand{\calR}{\mathcal{R}}
\newcommand{\calM}{\mathcal{M}}
\newcommand{\calS}{\mathcal{S}}
\newcommand{\calU}{\mathcal{U}}
\newcommand{\ratmap}{\dashedrightarrow}

\newcommand{\Aut}{\operatorname{Aut}}
\newcommand{\disc}{\operatorname{disc}}
\newcommand{\Gal}{\operatorname{Gal}}
\renewcommand{\Im}{\operatorname{Im}}

\newcommand{\Per}{\operatorname{Per}}
\newcommand{\PGL}{\operatorname{PGL}}
\newcommand{\Rat}{\operatorname{Rat}}

\newcommand{\Sym}{\operatorname{Sym}}

\newtheorem{thm}{Theorem}[section]

\newtheorem{lem}[thm]{Lemma}
\newtheorem{prop}[thm]{Proposition}
\newtheorem{cor}[thm]{Corollary}
\theoremstyle{definition}
\newtheorem{rem}[thm]{Remark}
\newtheorem{notation}[thm]{Notation}

\numberwithin{equation}{section}

\title[Dynatomic groups of quadratic rational maps]{Dynatomic Galois groups for a family of quadratic rational maps}

\author[D. Krumm]{David Krumm}
\email{david.krumm@gmail.com}
\urladdr{http://maths.dk}

\author[A. Lacy]{Allan Lacy}
\address{School of Mathematics, University of Costa Rica}
\email{allan.lacy@ucr.ac.cr}

\keywords{Dynatomic polynomials, periodic points, Galois specialization}
\subjclass{Primary 37P05; Secondary 11S20, 11G30}

\begin{document}
\maketitle

\begin{abstract}
For every nonconstant rational function $\phi\in\Q(x)$, the Galois groups of the dynatomic polynomials of $\phi$ encode various properties of $\phi$ that are of interest in the subject of arithmetic dynamics. We study here the structure of these Galois groups as $\phi$ varies in a particular one-parameter family of maps, namely the quadratic rational maps having a critical point of period 2. In particular, we provide explicit descriptions of the third and fourth dynatomic Galois groups for maps in this family.
\end{abstract}

\section{Introduction}\label{intro_section}

Let $k$ be a field of characteristic $0$ and $\phi\in k(x)$ a nonconstant rational function. The \textbf{degree} of $\phi$ is given by $\deg\phi=\max\{\deg p,\deg q\}$ if $p,q\in k[x]$ are coprime polynomials such that $\phi=p/q$. Corresponding to $\phi$ there is a morphism of algebraic varieties $\phi:\PP^1\to\PP^1$; fixing an algebraic closure $\bar k$ of $k$, this morphism gives rise to a map on point sets $\phi:\PP^1(\bar k)\to\PP^1(\bar k)$.

For every positive integer $n$, we denote by $\phi^n$ the $n$-fold composition of $\phi$ with itself. A point $P\in\PP^1(\bar k)$ is called \textbf{periodic} for $\phi$ if there exists $n\ge 1$ such that $\phi^n(P)=P$; in that case, the least such integer $n$ is the \textbf{period} of $P$, and $P$ is an \textbf{$n$-periodic point} of $\phi$. Letting $\Per_n(\phi)$ denote the (finite) set of all $n$-periodic points of $\phi$, the $n$th \textbf{dynatomic field} of $\phi$, denoted $k_{n,\phi}$, is the composite of all extensions $k(P)/k$ for $P\in\Per_n(\phi)$, where $k(P)$ denotes the field of definition of $P$. A simple argument shows that every automorphism of the extension $\bar k/k$ maps the set $\Per_n(\phi)$ to itself; thus $k_{n,\phi}/k$ is a Galois extension. The automorphism group $G_{n,\phi}:=\Gal(k_{n,\phi}/k)$ is the $n$th \textbf{dynatomic group} of $\phi$, and can be regarded as encoding various arithmetic-dynamical properties of the map $\phi$. Letting $r=(\#\Per_n(\phi))/n$, it is well known that $G_{n,\phi}$ can be embedded in the wreath product $(\Z/n\Z)\wr S_r$, where $S_r$ is the symmetric group of degree $r$ (see \cite{silverman}*{Theorem 3.56}).

\subsection{Families of dynatomic groups} In this paper we address the question of how the structure of the groups $G_{n,\phi}$ will change if the map $\phi$ is allowed to vary in a family of rational functions of fixed degree. A theorem of Morton \cite{morton_period 3}*{Theorem 8} provides an early example of a successful resolution of this type of question: for quadratic polynomials of the form $\phi(x)=x^2+a\in\Q[x]$ such that $-7-4a$ is not a square, Morton shows that the groups $G_{3,\phi}$ form a finite collection of known groups (up to isomorphism). A similar result for the groups $G_{4,\phi}$ of arbitrary quadratic polynomials $\phi$ is proved in \cite{krumm_fourth_dynatomic}. The main goal of the present paper is to classify the groups $G_{3,\phi}$ and $G_{4,\phi}$ for maps in a family of non-polynomial rational functions of degree 2. In order to define the maps of interest, we recall two concepts:
\begin{itemize}
\item A \textbf{critical point} of $\phi$ is a point $P\in\PP^1(\bar k)$ where the morphism $\phi:\PP^1\to\PP^1$ ramifies; equivalently, we have $\phi'(P)=0$ after applying a change of variables to move $P$ and $\phi(P)$ away from infinity.
\item For any field extension $K/k$, two rational functions $\phi, \psi\in k(x)$ are called \textbf{linearly conjugate} over $K$ if there exists $\sigma\in K(x)$ of degree 1 such that $\psi=\sigma^{-1}\circ\phi\circ\sigma$. In that case, $\phi$ and $\psi$ are considered to be equivalent as dynamical systems on $\PP^1(K)$.
\end{itemize}

As is well known to dynamicists, the collection of quadratic polynomial maps can be described, up to linear conjugacy over $\bar k$, as the collection of quadratic rational maps having a 1-periodic (i.e., fixed) critical point. Moreover, these maps form a one-parameter family represented by polynomials of the form $x^2+c$. In this paper we study the quadratic rational maps having a 2-periodic critical point, a collection of maps that is also known to form a one-parameter family up to conjugacy over $\bar k$ (see \cite{milnor}*{\textsection 3}).

\subsection{Main results} In stating our results we distinguish between rational maps according to their groups of symmetries. By definition, the \textbf{automorphism group} of a rational function $\phi\in k(x)$, denoted $\Aut(\phi)$, consists of all rational functions $\sigma\in\bar k(x)$ of degree 1 satisfying $\sigma^{-1}\circ\phi\circ\sigma=\phi$. 

The following theorems contain our results for the third dynatomic groups; analogous statements for the fourth dynatomic groups are proved in  \textsection\ref{Phi4_section}.

\begin{thm}\label{period2_main_thm_no_autos}
Let $\phi\in\Q(x)$ be a rational function of degree $2$ having a $2$-periodic critical point, and suppose that $\Aut(\phi)$ is trivial. Then $\phi$ is linearly conjugate over $\Q$ to a map of the form
\begin{equation}\label{no_auto_maps}
\phi_v:=\frac{v(x-1)}{x^2},\quad v\in\Q\setminus\{0\}.
\end{equation}
Moreover, with notation as in Section \ref{trivial_section_G3}, we have the following for $v\ne 3$: if $v$ is in the image of $\eta$, then $G_{3,\phi}\cong C$. Otherwise, $G_{3,\phi}\cong(\Z/3\Z)\wr S_2$.
\end{thm}

Note, in particular, that Theorem \ref{period2_main_thm_no_autos} explicitly identifies all groups that can be realized as dynatomic groups $G_{3,\phi}$ for maps $\phi$ as in the theorem.

\begin{thm}\label{period2_main_thm_autos}
Let $\phi\in\Q(x)$ be a rational function of degree $2$ having a $2$-periodic critical point, and suppose that $\Aut(\phi)$ is nontrivial. Then $\phi$ is linearly conjugate over $\Q$ to a map of the form
\begin{equation}\label{dun_maps}
\psi_v:=\frac{2x-1}{vx^2-1},\quad v\in\Q\setminus\{0,4\}.
\end{equation}
Moreover, with notation as in Section \ref{autos_section_G3}, we have the following for $v\ne 1$:
\begin{enumerate}[(a)]
\item If $v\notin\Im\alpha\cup\Im\beta$, then $G_{3,\phi}\cong W$;
\item If $v\in\Im\alpha$, then $G_{3,\phi}\cong A$, $K$, or $M$;
\item If $v\in\Im\beta$, then $G_{3,\phi}\cong B$ or $M$;
\item If $v\in\Im\mu$, then $G_{3,\phi}\cong M$;
\item If $v\in\Im\kappa$, then $G_{3,\phi}\cong K$.
\end{enumerate}
\end{thm}

Several arithmetic-dynamical properties of a rational map can be readily determined once the dynatomic groups of the map are well understood. The following theorem illustrates this fact by applying Theorem \ref{period2_main_thm_no_autos} to obtain information about the degrees of number fields generated by periodic points, as well as statistical data concerning periodic points over $p$-adic fields. We state our results only in the case of $3$-periodic points, but similar statements for 4-periodic points are proved in Section \ref{applications_section}.

\begin{thm}\label{intro_deg_dens_cor}
Let $\phi\in\Q(x)$ be a rational function of degree $2$ having a $2$-periodic critical point, and suppose that $\Aut(\phi)$ is trivial.
\begin{enumerate}[(a)]
\item\label{intro_deg_thm} The field of definition of every $3$-periodic point of $\phi$ is a number field of absolute degree $6$.
\item\label{intro_density_thm} For more than $72\%$ of primes $p$ (in the sense of Dirichlet density), $\phi$ does not have a $3$-periodic point in $\PP^1(\Q_p)$.
\end{enumerate}
\end{thm}

The above theorem and its period 4 counterpart provide new data relevant to one of the guiding problems in arithmetic dynamics, namely the uniform boundedness conjecture of Morton and Silverman \cite{morton-silverman}. For every nonconstant rational function $\phi\in\Q(x)$, let $\Per_n(\phi,\Q)$ denote the set of $n$-periodic points of $\phi$ that are defined over $\Q$. The Morton--Silverman conjecture would imply, in particular, the existence of a constant $B$ such that
\[\text{For all quadratic maps $\phi\in\Q(x)$, }\Per_n(\phi,\Q)=\emptyset \text{ if }n>  B.\] 
This weaker statement has not been proved, although significant progress has been made on studying rational periodic points of quadratic maps. Notably, Poonen \cite{poonen} provides evidence suggesting that $B=3$ suffices if $\phi$ is restricted to the family of quadratic \emph{polynomial} maps; see also \cite{flynn-poonen-schaefer,hutz-ingram,morton_period4,stoll}. Other families of quadratic maps are studied in \cite{canci-vishkautsan,lukas-manes-yap,manes,vishkautsan}. Our final result suggests that the bound $B=2$ suffices for maps as in Theorem \ref{intro_deg_dens_cor}. Noting that all maps of the form \eqref{no_auto_maps} have a 2-periodic critical point defined over $\Q$ (namely $0$), we deduce the following from an earlier result of Canci and Vishkautsan \cite{canci-vishkautsan}*{Theorem 1.3}.

\begin{cor}\label{intro_cycles_cor}For maps $\phi$ as in Theorem \ref{intro_deg_dens_cor}, we have
\[\Per_n(\phi,\Q)=\emptyset\;\text{ for } 3\le n\le6.\]
\end{cor}

\subsection{Methods} The core ingredient in the proofs of our main results is a technique for studying the Galois groups of one-variable specializations of a two-variable polynomial. Due to the details of this method (discussed in Section \ref{spec_section}), its applicability to the analysis of dynatomic groups is limited to one-parameter families of maps, such as the collections $\{\phi_v\}$ and $\{\psi_v\}$ of Theorems \ref{period2_main_thm_no_autos} and \ref{period2_main_thm_autos}. We end this introduction by highlighting other one-parameter families to which this method can in principle be applied.

As explained more precisely in Section \ref{M2_section}, work of Milnor \cite{milnor} shows that the set of all quadratic dynamical systems on $\PP^1$ is in bijection with the affine plane $\A^2$. From this point of view, quadratic polynomial maps form a line in the plane, and the same is true for maps having a 2-periodic critical point. Similarly, the collections of maps having a $3$-periodic or $4$-periodic critical point correspond to rational plane curves for which a parametrization can be easily computed. Hence, the analysis of dynatomic groups for maps in these families lies within the scope of the methods developed here. We leave the details of this analysis for the interested reader to explore.

\subsection{Outline} This paper is organized as follows. Section \ref{prelim_section} provides the necessary background material on arithmetic dynamics and discusses the technique of Galois specialization. The main results of the paper, namely Theorems \ref{period2_main_thm_no_autos} and \ref{period2_main_thm_autos} and their period 4 counterparts, are proved in Sections \ref{Phi3_section} and \ref{Phi4_section}. Theorem \ref{intro_deg_dens_cor} and analogous statements for period 4 are proved as applications in Section \ref{applications_section}.

\subsection{Computer code} All computations for this article were carried out using \textsc{Magma} \cite{magma} V2.26-12. The corresponding code is available in \cite{code}, with files named according to the relevant subsections of the article.

\section{Preliminaries}\label{prelim_section}

Throughout this section, we let $k$ be an arbitrary field of characteristic 0.

\subsection{Specialization of Galois groups}\label{spec_section} Let $t$ and $x$ be indeterminates, let $\theta\in k[t][x]$ be a nonconstant monic polynomial (in $x$), and let $G$ be the Galois group of $\theta$ over the function field $k(t)$. For every element $c\in k$ we may consider the specialized polynomial $\theta(c,x)\in k[x]$ and its Galois group, denoted $G_c$. The Hilbert Irreducibility Theorem \cite{serre_topics}*{Chapter 3} implies that $G_c\cong G$ for all $c$ outside a thin subset of $k$. The structure of $G_c$ for elements $c$ of the thin set can be studied by applying Proposition \ref{specialization_curves} below, which essentially yields a classification of the groups $G_c$ in terms of the sets of $k$-rational points on a collection of algebraic curves.

Recall the notion of an isomorphism of group actions: if $U$ is a group acting on a set $X$, and $V$ a group acting on a set $Y$, an isomorphism between the two group actions consists of a group isomorphism $f:U\to V$ and a bijection $s:X\to Y$ satisfying
\[s(g\cdot x)=f(g)\cdot s(x)\] for all $g\in U$ and $x\in X$. If such an isomorphism exists, we will write $U\equiv V$. Note, in particular, that if $U$ and $V$ are conjugate subgroups of the symmetric group $\Sym(X)$ (with its natural action on $X$), then $U\equiv V$.

If $F$ is a finite extension of the rational function field $k(t)$, by a \textbf{prime} of $F$ we mean a maximal ideal of the integral closure of $k[t]$ in $F$. For $c\in k$, we write $p_c$ to denote the prime $(t-c)$ of $k(t)$.

\begin{prop}\label{specialization_curves}
Let $\theta(t,x)\in k[t][x]$ be a monic polynomial (in $x$) with no repeated root, and let $\Delta$ be the set of roots of the discriminant of $\theta$. Let $E/k(t)$ be a splitting field for $\theta$ and $G=\Gal(E/k(t))$ the Galois group of $\theta$. Fix $c\in k\setminus\Delta$, and let $P$ be a prime of $E$ lying over $p_c$. Writing $D_P$ for the decomposition group of $P$ over $k(t)$, we have the following:
\begin{enumerate}
\item There is an isomorphism of group actions $G_c\equiv D_{P}$, where $G_c$ acts on the roots of $\theta(c,x)$ and $D_P$ acts on the roots of $\theta$.
\item Let $H$ be a subgroup of $G$ with fixed field $E^H$, and suppose that the extension $E^H/k(t)$ is generated by a root of the monic irreducible polynomial $q_H(t,x)\in k[t][x]$. The following equivalence holds if $c$ is not a root of the discriminant of $q_H$: 
\[(D_P\text{ is contained in a conjugate of  }H)\iff (q_H(c,x) \text{ has a root in } k).\]
\end{enumerate}
\end{prop}

\begin{proof}
This follows from Propositions 2.3 and 2.6 in \cite{krumm-sutherland}.
\end{proof}

\begin{rem} If the polynomial $\theta$ in Proposition \ref{specialization_curves} is not monic, the same statements hold with minor modifications. The condition $\disc\theta(c,x)\ne 0$ is strengthened to $\ell(c)\cdot\disc\theta(c,x)\ne 0$, where $\ell\in k[t]$ is the leading coefficient of $\theta$, and the condition $\disc q_H(c,x)\ne 0$ is replaced by $\ell(c)\cdot\disc q_H(c,x)\ne 0$.
\end{rem}

\begin{notation}\label{HIT_notation}
For every subgroup $H\subseteq G$, let us fix a choice of defining polynomial $q_H$ for the fixed field of $H$ (as in Proposition \ref{specialization_curves}). We denote by
\begin{itemize}
\item $Y_H$ the affine plane curve defined by $q_H(t,x)=0$; and
\item $\pi_H(k)$ the projection of the set $Y_H(k)$ onto the first coordinate.
\end{itemize}
\end{notation}

For an element $c\in k\setminus\Delta$ and a subgroup $H$ of $G$, Proposition \ref{specialization_curves} shows, loosely speaking, that $G_c$ is a subgroup of $H$ if and only if $c\in\pi_H(k)$; indeed, the latter condition is equivalent to the statement that $q_H(c,x)$ has a root in $k$. As a consequence, we will have $G_c=H$ if $c\in\pi_H(k)$ and $c\notin\pi_M(k)$ for every maximal subgroup $M$ of $H$. Thus, the sets of $k$-rational points $Y_H(k)$ and $Y_M(k)$ determine the specializations $G_c$ for which $G_c=H$. By allowing $H$ to vary among all subgroups of $G$, and determining the $k$-rational points on the relevant curves, we obtain a classification of the groups $G_c$ according to which subgroup of $G$ they are isomorphic to. This rough description of our method will be illustrated more precisely in subsequent sections.   

In the case $k=\Q$, much of the data needed to classify the groups $G_c$ can be explicitly computed: the Galois group $G$ can be obtained (as a permutation group) using \cite{krumm-sutherland}*{Algorithm 3.1}, and the polynomials $q_H$ can be found using \cite{krumm-sutherland}*{Algorithm 3.2}. Moreover, the (finitely many) groups $G_c$ with $c\in\Delta$ can be computed using available Galois group algorithms for polynomials over $\Q$; see \cite{fieker-kluners}. \textsc{Magma} includes implementations of all of the above algorithms. 

For curves $Y_H$ having finitely many rational points, one can attempt to determine the sets $Y_H(\Q)$ by applying several available methods for computing rational points on curves; see \cite{stoll_survey,poonen_survey} for a survey of current methods. If $Y_H(\Q)$ is infinite, a well-known theorem of Faltings implies that $Y_H$ is birational to either $\PP^1$ or an elliptic curve with positive Mordell--Weil rank. In the former case, a parametrization of $Y_H$ can be computed using the methods of \cite{sendra-winkler-diaz}*{Chapters 4-6}, while in the latter case $Y_H(\Q)$ can be described by providing a birational map to an elliptic curve in Weierstrass form and, if possible, a list of generators for the Mordell--Weil group of the curve.

\subsection{The moduli space $\calM_d$}\label{Md_section}
Let $d$ be a positive integer. For any intermediate field $L$ in the extension $\bar k/k$, let $\Rat_d(L)$ denote the set of rational functions in $L(x)$ of degree $d$. The set $\Rat_1(L)$ forms a group under composition, and as such it is isomorphic to $\PGL_2(L)$. The latter group may therefore be regarded as acting on $\Rat_d(L)$ by conjugation; for $\phi\in\Rat_d(L)$ and $\sigma\in\PGL_2(L)$, we write $\phi^{\sigma}:=\sigma^{-1}\circ\phi\circ\sigma$. 

For any map $\phi\in\Rat_d(\bar k)$, let $[\phi]=\{\phi^{\sigma}:\sigma\in\PGL_2(\bar k)\}$. Following Silverman \cite{silverman_fod}, we call $[\phi]$ the \textbf{dynamical system} of $\phi$. Thus, elements of the orbit space
\[\Rat_d(\bar k)/\PGL_2(\bar k)=\{[\phi]:\phi\in\Rat_d(\bar k)\}\]
are dynamical systems of degree $d$.
 As shown by Silverman \cite{silverman_moduli}  (see also \cite{silverman}*{\textsection 4.4}), there exists an affine algebraic variety $\calM_d$, defined over $k$, with the property that there is a bijection
\[\calM_d(\bar k)\stackrel{\sim}{\longrightarrow}\Rat_d(\bar k)/\PGL_2(\bar k),\]
so that points on $\calM_d$ correspond to dynamical systems of degree $d$.

Multipliers at periodic points, defined below, are useful for analyzing the varieties $\calM_d$. If $P\ne\infty$ is an $n$-periodic point of the map $\phi\in\bar k(x)$, the \textbf{multiplier} of $\phi$ at $P$ is given by $\lambda_{\phi,P}:=(\phi^n)'(P)$. If $\infty\in\Per_n(\phi)$, we set
$\lambda_{\phi,\infty}:=\lambda_{\phi^{\sigma},P}$ for any $\sigma\in\PGL_2(\bar k)$ and $P\in\PP^1(\bar k)$ satisfying $\sigma(P)=\infty$ and $P\ne\infty$. (The fact that $\lambda_{\phi,\infty}$ is well defined follows from \cite{silverman}*{Proposition 1.9}.) With these definitions, we have the following basic property of multipliers:
\begin{equation}\label{multiplier_conjugate}
\lambda_{\phi,P}=\lambda_{\phi^{\sigma},\sigma^{-1}(P)}
\end{equation}

for $\sigma\in\PGL_2(\bar k)$ and $P\in\Per_n(\phi)$.

The $n$-\textbf{multiplier spectrum} of $\phi$ is the multiset
\[\Lambda_n(\phi)=\{\lambda_{\phi,P}:P\in\Per_n(\phi)\}.\]
It follows from \eqref{multiplier_conjugate} that linearly conjugate maps have equal multiplier spectra; hence, the $n$-multiplier spectrum is a well-defined invariant of the dynamical system $[\phi]\in\calM_d(\bar k)$. The following lemma shows that multiplier spectra can be used to detect the existence of periodic critical points.

\begin{lem}\label{mult0_lem}A dynamical system $[\phi]\in\calM_d(\bar k)$ has an $n$-periodic critical point if and only if $0\in\Lambda_n(\phi)$.
\end{lem}
\begin{proof}
Conjugating $\phi$ if necessary, we may assume that $\infty\notin\Per_n(\phi)$. Let $P\in\Per_n(\phi)$ and let $C$ denote the $\phi$-orbit of $P$, i.e., the set $\{\phi^i(P):i\ge 0\}$. Since $\infty\notin C$, the chain rule yields
\begin{equation}\label{mult_ident}
\lambda_{\phi,P}=(\phi^n)'(P)=\prod_{i=0}^{n-1}\phi'(\phi^i(P))=\prod_{\beta\in C}\phi'(\beta).
\end{equation}
Suppose that $P$ is a critical point of $\phi$. Then $\phi'(P)=0$, and \eqref{mult_ident} implies $\lambda_{\phi,P}=0$. Thus $0\in\Lambda_n(\phi)$. Conversely, if $\lambda_{\phi,P}=0$, then \eqref{mult_ident} implies that $\phi'(\beta)=0$ for some $\beta\in C$. Thus $\beta$ is an $n$-periodic critical point of $\phi$.
\end{proof}

\subsection{The moduli space $\calM_2$}\label{M2_section} By considering multipliers at fixed points, Milnor \cite{milnor} showed that the variety $\calM_2$ is isomorphic to the affine plane over $k$ (see also \cite{silverman}*{Theorem 4.56}). Explicitly, there is an isomorphism
\begin{equation}\label{M2A2_iso}
\calM_2\stackrel{\sim}{\longrightarrow}\A^2,\quad[\phi]\mapsto(\sigma_1,\sigma_2),
\end{equation}
 where $\sigma_i$ is the $i$th elementary symmetric function of the elements of $\Lambda_1(\phi)$. We may therefore regard points of $\A^2$ as corresponding to quadratic dynamical systems on $\PP^1$. Taking this point of view, Milnor notes that various subsets of $\calM_2$ that are defined by dynamical properties correspond to algebraic curves in the plane. In particular, the set of quadratic dynamical systems having an $n$-periodic critical point corresponds to a curve which we denote by $C_n$.\footnote{Milnor uses the notation $\Per_n(0)$ for our curve $C_n$, and defines this curve as the set of dynamical systems $[\phi]\in\calM_2(\bar k)$ having an $n$-periodic point with multiplier 0. This definition is equivalent to the one given here, by Lemma \ref{mult0_lem}.} In addition, the \textbf{symmetry locus} in $\calM_2$, i.e., the set of dynamical systems $[\phi]$ for which the group $\Aut(\phi)$ is nontrivial, corresponds to a curve which we denote by $\calS$. For proofs that $C_n$ and $\calS$ are one-dimensional, see Corollaries D.2 and 5.3 in Milnor's article.

Let us use $(r,s)$ as coordinates on $\A^2$. Milnor \cite{milnor}*{Corollary 5.3} obtains a parametrization of $\calS$ from which we deduce the following equation for $\calS$:
\begin{equation}\label{symlocuseq}
-2r^3-r^2s+r^2+8rs+4s^2-12r-12s+36=0.
\end{equation}
Furthermore, Milnor \cite{milnor}*{Appendix D} provides the equation $r=2$ for the curve $C_1$ (which corresponds to the collection of quadratic polynomial maps), and the equation $s=-2r$ for the curve $C_2$. Though we will not work here with the curves $C_3$ and $C_4$, we remark that both of these are rational plane curves, and Milnor's article includes an equation for $C_3$.

In what follows, we identify $\calM_2$ with $\A^2$ using the isomorphism \eqref{M2A2_iso}. The next two lemmas provide one-parameter descriptions of the $\Q$-linear conjugacy classes of dynamical systems in $C_2(\Q)$.  

\begin{lem}\label{C2param}
Suppose that $\phi\in\Rat_2(\Q)$ has a $2$-periodic critical point, and that $\Aut(\phi)$ is trivial. Then $\phi$ is linearly conjugate over $\Q$ to a map of the form $\phi_v$ defined in \eqref{no_auto_maps}.
\end{lem}
\begin{proof}
 Let $(r,s)$ be the point of $\A^2$ corresponding to $[\phi]$ under the isomorphism \eqref{M2A2_iso}. By \cite{manes-yasufuku}*{Lemma 3.1}, we have $[\phi]=[\psi_{r,s}]$, where
\[\psi_{r,s}:=\frac{2x^2+(2-r)x+(2-r)}{-x^2+(2+r)x+2-r-s}\in\Q(x).\]
Moreover, the fact that $\Aut(\phi)$ is trivial implies that $\phi$ and $\psi_{r,s}$ are linearly conjugate over $\Q$ (see \cite{silverman}*{Proposition 4.73}).

Since $\phi$ has a 2-periodic critical point, the point $(r,s)$ lies on the curve $C_2$, so $s=-2r$. Thus, $\phi$ is linearly conjugate to the map $f:=\psi_{r,-2r}$. Letting $\sigma(x)=(2-x)/(x-1)$, we compute that $\sigma^{-1}\circ f\circ\sigma=(r+6)(x-1)/x^2=\phi_v$, where $v=r+6$. We must have $v\ne 0$ since otherwise $f$ would be constant.
\end{proof}

\begin{lem}\label{dunaiskylemma} 
Suppose that $\phi\in\Rat_2(\Q)$ has a $2$-periodic critical point, and that $\Aut(\phi)$ is nontrivial. Then $\phi$ is conjugate over $\Q$ to a map of the form $\psi_v$ defined in \eqref{dun_maps}.
\end{lem}
\begin{proof}
We begin by showing that there is a unique point in $C_2(\Q)\cap\calS(\Q)$, namely $[1/x^2]\in\calM_2(\Q)$. Using the equations $s=-2r$ for $C_2$ and \eqref{symlocuseq} for $\calS$, a simple calculation shows that $C_2(\Q)\cap\calS(\Q)=\{(-6,12)\}$. For the map $f=1/x^2$ we have $\Lambda_1(f)=\{-2,-2,-2\}$, so $(-6,12)$ is the point corresponding to $[f]$ under the isomorphism $\calM_2\to\A^2$. This proves the claim.

For the remainder of the proof we use an argument sketched in \cite{dunaisky}*{Theorem 5.2.1}. Conjugating $\phi$ by an element of $\PGL_2(\Q)$, we may assume that $\phi(\infty)=0$ and $\phi(0)=1$, from which it follows that $\phi$ has the form
\[\phi(x)=\frac{ax+b}{x^2+cx + b}\]
with $a,b,c\in\Q$ and $b\ne 0$. The hypotheses imply that $[\phi]\in C_2(\Q)\cap\calS(\Q)$, so $[\phi]=[1/x^2]$, and therefore $\Lambda_1(\phi)=\{-2,-2,-2\}$. Imposing the latter condition by direct computation, we will show that $a=-2b$ and $c=0$, putting $\phi$ in the desired form $\psi_v$ with $v=-1/b$. Clearly $v\ne 0$, and also $v\ne 4$ since otherwise $\deg\psi_v=1$.

We compute
\[x-\phi(x)=\frac{p(x)}{x^2+cx+b},\quad 2+\phi'(x)=\frac{q(x)}{(x^2+cx+b)^2},\]

where
\begin{align*}
p(x)&=x^3+cx^2+(b-a)x-b,\\
q(x)&=(2x+2c)\cdot p(x) + (a + 2b)x^2 + (2ac + 2bc)x + ab + 2b^2 + bc.
\end{align*}

Since $1/x^2$ has three distinct fixed points, each with multiplier $-2$, the same must hold for $\phi$. It follows that $p$ has three distinct roots, each of which is also a root of $q$; hence $p$ divides $q$. The definitions now imply that
\[(a + 2b)x^2 + (2ac + 2bc)x + ab + 2b^2 + bc=0,\]
leading to $a=-2b$ and $c=0$, as desired.
\end{proof}

\subsection{Dynatomic polynomials}
For a nonconstant rational map $\phi\in k(x)$ and a positive integer $n$, we set

\begin{equation}\label{dynatomic_poly_defn}
\Phi_{n,\phi}:=\prod_{d|n}(x\cdot q_d-p_d)^{\mu(n/d)},
\end{equation}
where $\mu$ is the classical M{\"o}bius function and $p_d,q_d\in k[x]$ are coprime polynomials such that $\phi^d=p_d/q_d$. The rational function defined by \eqref{dynatomic_poly_defn} is in fact a polynomial, called the $n$th \textbf{dynatomic polynomial} of $\phi$; see \cite{silverman}*{\textsection 4.1} for proofs of this and other basic properties of dynatomic polynomials.

The key feature of $\Phi_{n,\phi}$ for our purposes is its relation to $n$-periodic points: writing $\PP^1(\bar k)=\bar k\cup\{\infty\}$ with $\infty=[1,0]$, every $n$-periodic point of $\phi$ in $\bar k$ is a root of $\Phi_{n,\phi}$, and the converse holds if $\Phi_{n,\phi}$ has no repeated root. Using this fact, we obtain a description of dynatomic groups as Galois groups.

\begin{lem}\label{dynatomic_splitting_field}
Suppose that the polynomial  $\Phi_{n,\phi}$ has nonzero discriminant. Then the dynatomic group $G_{n,\phi}$ is the Galois group of $\Phi_{n,\phi}$.
\end{lem}
\begin{proof}
It suffices to show that the dynatomic field $k_{n,\phi}$ is a splitting field for $\Phi_{n,\phi}$ over $k$. Since $k(\infty)=k$, the extension $k_{n,\phi}/k$ is generated by the set of $n$-periodic points of $\phi$ in $\bar k$, which is equal to the set of roots of $\Phi_{n,\phi}$.
\end{proof}

We end this section with a useful lemma for purposes of classifying the dynatomic groups of a family of rational maps.

\begin{lem}\label{Galois_iso_lem}
Suppose that the rational functions $\phi,\psi\in k(x)$ are linearly conjugate over $k$. Then $G_{n,\phi}=G_{n,\psi}$ for all $n\ge 1$.
\end{lem}
\begin{proof}
By definition, $G_{n,\phi}=\Gal(k_{n,\phi}/k)$, where $k_{n,\phi}=k(\Per_n(\phi))$. Let $s\in k(x)$ be a map of degree 1 satisfying $s^{-1}\circ\phi\circ s=\psi$. Then $s$ gives rise to a bijection $\Per_n(\psi)\to\Per_n(\phi)$ with the property that $k(s(P))=k(P)$ for all $P\in\Per_n(\psi)$. Therefore, $k(\Per_n(\psi))=k(\Per_n(\phi))$ and $k_{n,\psi}=k_{n,\phi}$.
\end{proof}

\section{The third dynatomic group}\label{Phi3_section}
In this section we apply the method of Galois specialization (Section \ref{spec_section}) to determine all the groups that can be realized as dynatomic groups $G_{3,\phi}$ for maps $\phi\in\Rat_2(\Q)$ having a 2-periodic critical point. For this purpose, it will be convenient to distinguish between dynamical systems according to their automorphism groups; our main results are Proposition \ref{C2G3} in the generic case of trivial automorphism group, and Proposition \ref{C2G3_autos} in the special case of maps with nontrivial automorphisms.

\begin{notation} For subgroups $U,V$ of a group $G$, we write $U\le V$ if $U$ is contained in a conjugate of $V$ by an element of $G$. If $\rho\in\Q(x)$ is a rational function, we write $\Im\rho$ for the image under $\rho$ of the set of all rational numbers where $\rho$ is defined. Throughout the remainder of the article, we use the notation $q_H,\pi_H, Y_H$ introduced in Notation \ref{HIT_notation}.
\end{notation}

\subsection{Maps with trivial automorphism group}\label{trivial_section_G3} In view of Lemmas \ref{C2param} and \ref{Galois_iso_lem}, our main task is to classify the groups $G_{3,v}:=G_{3,\phi_v}$ with $v\in\Q\setminus\{0\}$. The following are a few preliminary calculations needed to obtain such a classification.

Let $t$ and $x$ be indeterminates, and let $\phi\in\Q(t)(x)$ be defined by
\begin{equation}\label{genphi}
\phi=\frac{t(x-1)}{x^2}.
\end{equation}

A calculation of the dynatomic polynomial $\Phi_3=\Phi_{3,\phi}\in\Q(t)[x]$ yields
\begin{align*}
\Phi_3(t,x)=x^6 - 2tx^5 + (t^2 + 3t)x^4 + (-3t^2 - t)x^3 + 4t^2x^2 - 3t^2x + t^2.
\end{align*}

The discriminant of $\Phi_3$ is given by $\disc\Phi_3=-27t^{10}(t^2 - 9t + 27)^2$. It follows that, for $v\in\Q\setminus\{0\}$, the polynomial $\Phi_{3,\phi_v}(x)=\Phi_3(v,x)$ has nonzero discriminant; hence, by Lemma \ref{dynatomic_splitting_field}, $G_{3,v}$ is the Galois group of $\Phi_3(v,x)$.

Let $E/\Q(t)$ be a splitting field for $\Phi_3$, and $G=\Gal(E/\Q(t))$ the Galois group of $\Phi_3$. Computing a permutation representation of $G$, we find that $G\equiv W$, where $W$ is the centralizer of the permutation $(1,2,3)(4,5,6)$ in the symmetric group $S_6$. Note that $W$ factors as a wreath product $(\Z/3\Z)\wr S_2$.

Computing the lattice of subgroups of $G$, we find that $G$ has three maximal subgroups up to conjugacy, namely
\begin{align*}
A &= \langle(1,2,3),\; (1,2,3)(4,5,6)\rangle,\\
B &= \langle(1,5,2,6,3,4),\; (1,2,3)(4,5,6)\rangle,\\
C &= \langle(1,6)(2,4)(3,5),\; (1,2,3)(4,6,5)\rangle.
\end{align*}

The following are defining polynomials for the fixed fields of $A,B,C$:
\begin{align*}
q_A &= x^2 + (2t^2 - 6t + 12)x + t^4 - 6t^3 + 24t^2 - 36t + 36,\\
q_B &= x^3 + (8t^2 - 18t)x^2 + (20t^4 - 84t^3 + 72t^2)x\; +\\
&\quad 16t^6 - 96t^5 + 
    168t^4 - 144t^3 + 216t^2,\\
q_C &= x^3 + (8t^2 - 18t)x^2 + (20t^4 - 84t^3 + 72t^2)x\; + \\&\quad16t^6 - 96t^5 + 
    160t^4 - 72t^3.
\end{align*}

Up to conjugacy in $G$, the group $C$ has two maximal subgroups, namely,
\[H=\langle(1, 6)(2, 4)(3, 5)\rangle\quad\text{and}\quad J=\langle(1, 2, 3)(4, 6, 5)\rangle.\]
Moreover, one can verify the relations $H\le B$ and $J\le A$. 

\begin{lem}\label{G3curves} With groups $A,B,C$ as above, the following hold.
\begin{enumerate}[(a)]
\item\label{G3crva} $Y_A(\Q)=\{(0,-6)\}$.
\item\label{G3crvb} $Y_B(\Q)=\{(0,0),(3,0),(3, -18)\}$.
\item\label{G3crvc} Let $\eta\in\Q(t)$ be defined by
\begin{equation}\label{etamap}
\eta(t) = \frac{t^3 + 3t^2 - 6t + 1}{t(t-1)}.
\end{equation}
There is a rational function $\lambda\in\Q(t)$ such that
\[Y_C(\Q)=\{(0,0)\}\cup\{(\eta(v),\lambda(v)): v\in\Q\setminus\{0,1\}\}.\]
\end{enumerate}
\end{lem}
\begin{proof}
Since $q_A$ is quadratic in $x$, $Y_A$ is isomorphic to the affine curve
\[y^2=\disc q_A = -12t^2.\]
The latter curve has a unique rational point, namely $(0,0)$, so \eqref{G3crva} follows.

We compute that $Y_B$ has genus 1. Using the nonsingular point $(3,-18)$ on $Y_B$, we obtain a birational map $Y_B\ratmap E$ defined over $\Q$, where $E$ is the elliptic curve with Cremona label 27a1. The latter curve has rank 0 and torsion subgroup of order 3. Now \eqref{G3crvb} follows by pulling back the rational points on $E$ to points on $Y_B$.

The curve $Y_C$ has genus 0. Using the nonsingular point $(9/2,-27)$ on $Y_C$, we compute a birational map $\A^1\ratmap Y_C$ given by $v\mapsto(\eta(v),\lambda(v))$, where $\eta$ is defined as in \eqref{etamap}. Computing an inverse map $Y_C\ratmap\A^1$, we find that this map is defined at every rational point other than $(0,0)$. This proves \eqref{G3crvc}. 
\end{proof}

We now prove the main result for this section. Together with Lemma \ref{C2param}, this will complete the proof of Theorem \ref{period2_main_thm_no_autos}.

\begin{prop}\label{C2G3} 
Let $v\in\Q\setminus\{0,3\}$ and let $\eta$ be the map \eqref{etamap}.
\begin{enumerate}[(a)]
\item\label{C2G3a} If $v\notin\Im(\eta)$, then $G_{3,v}\equiv W$.
\item\label{C2G3b} If $v\in\Im(\eta)$, then $G_{3,v}\equiv C$.
\end{enumerate}
\end{prop}
\begin{proof}
Recall that $G_{3,v}$ is the Galois group of $\Phi_3(v,x)$. By Proposition \ref{specialization_curves} applied to $\theta=\Phi_3$, we have $G_{3,v}\equiv D_P$, where $P$ is any prime of the splitting field $E/\Q(t)$ lying over the prime $(t-v)$ of $\Q(t)$.

Suppose first that $v\notin\Im(\eta)$. Then we claim that $D_P=G$. Otherwise, $D_P$ is contained in a maximal subgroup of $G$, which must be conjugate to either $A$, $B$, or $C$. By Proposition \ref{specialization_curves}, one of the polynomials $q_A(v,x), q_B(v,x), q_C(v,x)$ must have a rational root. (The only rational roots of the discriminants of $q_A,q_B,q_C$ are $0$ and $3$, both of which are excluded from consideration.) Thus $v$ is the first coordinate of a rational point on either $Y_A$, $Y_B$, or $Y_C$.  Now Lemma \ref{G3curves} implies that
\[v\in\pi_A(\Q)\cup\pi_B(\Q)\cup\pi_C(\Q)=\{0,3\}\cup\Im(\eta),\]
a contradiction. Hence, $D_P=G$ and therefore $G_{3,v}\equiv G\equiv W$, proving \eqref{C2G3a}. 

Now suppose that $v\in\Im(\eta)$. Then $v\in\pi_C(\Q)$, so $D_P\le C$. Replacing $P$ by a conjugate prime if necessary, we may assume that $D_P\subseteq C$. We claim that $D_P=C$, which will complete the proof of the proposition.

Assume that $D_P$ is a proper subgroup of $C$. Since $H$ and $J$ are the only maximal subgroups of $C$ (up to $G$-conjugacy), we must have $D_P\le H$ or $D_P\le J$. However, since $H\le B$ and $J\le A$, this implies $D_P\le B$ or $D_P\le A$, and thus $v\in\pi_B(\Q)\cup\pi_A(\Q)=\{0,3\}$, a contradiction. Hence $D_P=C$.
\end{proof}

\subsection{Maps with nontrivial automorphisms}\label{autos_section_G3} In this section we classify the groups $G_{3,\phi}$ for maps $\phi\in\Rat_2(\Q)$ having a 2-periodic critical point and nontrivial automorphism group. In view of Lemma \ref{dunaiskylemma}, it suffices to classify the groups $G_{3,v}:=G_{3,\psi_v}$ with $v\in\Q\setminus\{0,4\}$. The overall process will be similar to that of Section \ref{trivial_section_G3}, with differences in the methods used to find rational points on curves.

For the rational function
\begin{equation}\label{genpsi}
\psi:=\frac{2x-1}{tx^2-1}\in\Q(t)(x),
\end{equation}

the third dynatomic polynomial $\Phi_3=\Phi_{3,\psi}\in\Q(t)[x]$ is given by

\begin{align*}
\Phi_3(t,x)=(t^4 - t^3)x^6 - 9t^3x^5 + (3t^3 + 33t^2)x^4 + (-t^3 - 26t^2)x^3 + \\
    (18t^2 - 27t)x^2 + (-6t^2 + 15t)x + t^2 - 4t + 3.
\end{align*}

For $v\in\Q\setminus\{0,4\}$ the discriminant of $\Phi_3(v,x)$ is nonzero, so the dynatomic group $G_{3,v}$ is the Galois group of $\Phi_3(v,x)$, by Lemma \ref{dynatomic_splitting_field}.  

Letting $E/\Q(t)$ be a splitting field for $\Phi_3$ and $G=\Gal(E/\Q(t))$, a Galois group computation shows that $G\equiv W$, where $W=(\Z/3\Z)\wr S_2$. For the analysis in this section we will need to consider several subgroups of $G$, denoted
$A,B,C,H,J,K,M$. Up to conjugacy in $G$, the groups $A,B,C$ are the maximal subgroups of $G$, and are uniquely determined by the following properties: $A$ has order 9, $B$ is cyclic of order 6, and $C$ is non-cyclic of order 6. The groups $J,K,M$ are the maximal subgroups of $A$, with $J$ uniquely determined by the condition $J\le C$, and $M$ by the condition $M\le B$. The group $H$ is the unique subgroup of $G$ (up to conjugacy) having order 2. The following relations are easily verified:

\begin{itemize}
\item $J\le C$ and $M\le B$;
\item $H\le B$ and $H\le C$;
\item The maximal subgroups of $B$ are conjugate to $H, M$;
\item The maximal subgroups of $C$ are conjugate to $H, J$.
\end{itemize}

Computing the fixed fields of the above subgroups of $G$, we obtain defining polynomials $q_A,q_B,q_C,\ldots, q_M$. For $v\in\Q\setminus\{0,1,4\}$, all of the specialized polynomials $q(v,x)$, as well as $\Phi_3(v,x)$, have nonzero discriminant. The first two $q$-polynomials are given by
\begin{align*}
q_A &= x^2 + 9t^3x + 27t^6 - 27t^5,\\
q_B &= x^3 + (-18t^7 + 117t^6 + 144t^5)x^2 + (-324t^{13} - 729t^{12} + 15552t^{11} + 
    5184t^{10})x \\+ & 216t^{22}- 3240t^{21} + 18792t^{20} - 54378t^{19} + 59859t^{18} -
    16848t^{17} + 527040t^{16}.
\end{align*}

\begin{lem}\label{G3curves_auto} With notation as above, the following hold.
\begin{enumerate}[(a)]
\item\label{G3crva} $\pi_A(\Q)=\Im\alpha$, where $\alpha(t)=(1 + t)^2/(1 - t + t^2)$.
\item\label{G3crvb} $\pi_B(\Q)=\{1\}\cup\Im\beta$, where $\beta(t)=t^2(3 - t)$.
\item\label{G3crvc} $\pi_C(\Q)=\{0,1,4\}$.
\item\label{G3crvk} $\pi_K(\Q)=\{0,4\}\cup\Im\kappa$, where
\[\kappa(t)=\frac{(t^3 - 3t^2 - 6t - 1)^2}{(1 + t + t^2)^3}.\]
\item\label{G3crvm} $\pi_M(\Q)\supseteq\Im\mu$, where $\mu(t)=\alpha\left((t^3 - 3t - 1)/(t^3 + 3t^2 - 1)\right)$.
\end{enumerate}
\end{lem}

\begin{proof}
The curve $Y_A$ has genus 0. Using the nonsingular point $(1,0)$ on $Y_A$, we compute a birational map $\A^1\ratmap Y_A$ given by $v\mapsto (\alpha(v),z(v))$ for some rational function $z$. This map is defined everywhere on $\A^1(\Q)$, so $\Im\alpha\subseteq\pi_A(\Q)$. Computing an inverse, we obtain a map $Y_A\ratmap\A^1$ defined at every rational point different from $(0, 0), (1, -9)$, and $(1, 0)$. It follows that $\pi_A(\Q)\setminus\Im\alpha\subseteq\{0,1\}$.
Moreover, direct computation shows that $0,1\in\Im\alpha$. This proves \eqref{G3crva}, and parts \eqref{G3crvb} and \eqref{G3crvk} are proved similarly.

For the proof of \eqref{G3crvc} we use the methods described in  \cite{fieker-sutherland}*{\textsection 6} to factor the polynomial $q_C$ over $\overline\Q$. Letting $L$ be the cubic number field generated by a root $g$ of the polynomial $x^3 - 3x + 1$, we have the  factorization
\[q_C(t,x)=(u_2g^2 + u_1g + u_0)(v_2g^2 + v_1g + v_0)(w_2g^2 + w_1g + w_0),\]

in the ring $L[t][x]$, where $u_2(t,x)=w_1(t,x)=0$ and
\begin{align*}
u_1(t,x)&=6t^7 - 30t^6 + 24t^5,\\
u_0(t,x)&=-6t^7 + 39t^6 + 48t^5 + x,\\
v_2(t,x)&=-6t^7 + 30t^6 - 24t^5,\\
v_1(t,x)&=v_2(t,x),\\
v_0(t,x)&=6t^7 - 21t^6 + 96t^5 + x,\\
w_2(t,x)&=-v_2(t,x),\\
w_0(t,x)&=-18t^7 + 99t^6 + x.\\
\end{align*}

\vspace{-5mm}
Note that all the polynomials displayed above have rational coefficients. For any point $(a,b)\in Y_C(\Q)$ we have $q_C(a,b)=0$, so the given factorization of $q_C(t,x)$ implies that one of the following equalities must hold:
\begin{align*}
(u_2(a,b),u_1(a,b),u_0(a,b))&=(0,0,0),\\
(v_2(a,b),v_1(a,b),v_0(a,b))&=(0,0,0),\\
(w_2(a,b),w_1(a,b),w_0(a,b))&=(0,0,0).
\end{align*}
Each of these three equalities defines a 0-dimensional subscheme of the affine plane, whose rational points can therefore be computed. For all three schemes we obtain the same set of rational points, namely
\[\{(0, 0), (1, -81), (4, -110592)\}.\] 
Now \eqref{G3crvc} now follows immediately.

Finally, to prove \eqref{G3crvm} we parametrize the curve $Y_M$. A search for a nonsingular rational point of small height on $Y_M$ fails to find any such point, so an indirect approach is necessary. Given that $M\le A$ and $M\le B$, we have
\[\pi_M(\Q)\subseteq (\pi_A(\Q)\cap\pi_B(\Q))\cup\{0,1,4\},\]
which suggests using the known parametrizations of $Y_A$ and $Y_B$. Let $\alpha, i$ be the components of the map $\A^1\ratmap Y_A$, and $\beta,j$ the components of $\A^1\ratmap Y_B$. The equation $\alpha(x)=\beta(y)$ defines, after clearing denominators, an affine plane curve which we denote by $Y$. This curve has genus 0 and is easily parametrized; we thus obtain a birational map $\A^1\ratmap Y$ with first coordinate
\[m(v)=\frac{v^3 - 3v - 1}{v^3 + 3v^2 - 1}.\]
Defining $\mu=\alpha\circ m$, we expect to find a parametrization of $Y_M$ whose first coordinate is $\mu$. With $v$ as an indeterminate, we factor the polynomial $q_M(\mu(v),x)\in\Q(v)[x]$ and find that it has six roots in $\Q(v)$; choosing a root $p$, we obtain have a parametrization $\A^1\ratmap Y_M$ given by $v\mapsto (\mu(v),p(v))$. This map is defined everywhere on $\A^1(\Q)$, so \eqref{G3crvm} follows.
\end{proof}

The next proposition, together with Lemma \ref{dunaiskylemma}, completes the proof of Theorem \ref{period2_main_thm_autos}.

\begin{prop}\label{C2G3_autos}
The following statements hold for $v\in\Q\setminus\{0,1,4\}$.
\begin{enumerate}[(a)]
\item\label{C2G3autosm} If $v\in\Im\mu$, then $G_{3,v}\equiv M$.
\item\label{C2G3autosk} If $v\in\Im\kappa$, then $G_{3,v}\equiv K$.
\item\label{C2G3autosw} If $v\notin\Im\alpha\cup\Im\beta$, then $G_{3,v}\equiv W$.
\item\label{C2G3autosam} If $v\in\Im\alpha$, then $G_{3,v}\equiv A$, $K$, or $M$.
\item\label{C2G3autosbm} If $v\in\Im\beta$, then $G_{3,v}\equiv B$ or $M$.
\end{enumerate}
\end{prop}
\begin{proof} Throughout, we let $P$ be any prime of the splitting field $E$ lying over the prime $(t-v)$ of $\Q(t)$. By Proposition \ref{specialization_curves}, we have $G_{3,v}\equiv D_P$. Moreover, given that $\pi_C(\Q)=\{0,1,4\}$ and $v$ is not in this set, we have $D_P\not\le C$.

Suppose first that $v\in\Im\mu$. Then $v\in\pi_M(\Q)$ by Lemma \ref{G3curves_auto}, so $D_P\le M$, and we may assume $D_P\subseteq M$. We claim that $D_P=M$. The group $M$ has order 3, so its only proper subgroup is the trivial group. If $D_P=\{1\}$, then $D_P\le C$, which is a contradiction. Thus $D_P=M$ and $G_{3,v}\equiv M$, proving \eqref{C2G3autosm}. A very similar argument proves \eqref{C2G3autosk}, as $K$ also has order 3.

Suppose now that $v\notin\Im\alpha\cup\Im\beta$. Then $v\notin\pi_A(\Q)\cup\pi_B(\Q)$ by Lemma \ref{G3curves_auto}, so $D_P\not\le A$ and $D_P\not\le B$. Moreover, we have already noted that $D_P\not\le C$. Therefore, $D_P=G\equiv W$ since $A,B,C$ are the only maximal subgroups of $G$, up to conjugacy. This proves \eqref{C2G3autosw}.

Next, suppose that $v\in\Im\alpha$, so that  we may assume $D_P\subseteq A$. The maximal subgroups of $A$ are $J,K,M$, so if $D_P\ne A$, then $D_P\le J, K$, or $M$. However, since $J\le C$, the case $D_P\le J$ would lead to a contradiction. Thus $D_P\le K$ or $M$. Moreover, the proofs of \eqref{C2G3autosm} and \eqref{C2G3autosk} show that $D_P$ cannot be (conjugate to) a proper subgroup of $K$ or $M$. This proves \eqref{C2G3autosam}, and \eqref{C2G3autosbm} follows similarly, using the facts that the maximal subgroups of $B$ are $H$ and $M$, and $H\le C$.
\end{proof}

\section{The fourth dynatomic group}\label{Phi4_section}
In this section we study the groups $G_{4,\phi}$ for maps $\phi\in\Rat_2(\Q)$ having a 2-periodic critical point. Our main results are Propositions \ref{C2G4_autos} and \ref{C2G4}. As in Section \ref{Phi3_section}, we treat separately the cases where $\phi$ has trivial and nontrivial automorphism group; we begin with the latter case, which is simpler.

\subsection{Maps with nontrivial automorphisms}\label{autos_section_G4}
For maps $\psi_v$ as in \eqref{dun_maps}, we aim to classify the groups $G_{4,v}:=G_{4,\psi_v}$. Thus, we study the specializations of the fourth dynatomic Galois group of the rational map $\psi$ defined in \eqref{genpsi}.

Let $W$ denote the wreath product $(\Z/4\Z)\wr S_3$, realized as the centralizer in $S_{12}$ of the permutation $(1,2,3,4)(5,6,7,8)(9,10,11,12)$. Up to conjugacy, $W$ has a unique subgroup, denoted $V$, that is isomorphic to $(\Z/4\Z)\times S_3$. We compute that the polynomial $\Phi_4:=\Phi_{4,\psi}$ has Galois group $G\equiv V$. 

We now label several relevant subgroups of $G$. Up to conjugacy, $G$ has four maximal subgroups, which we denote by $A,B,C,D$. These groups are uniquely determined by the following properties: $A$ is cyclic of order 12; $B$ is isomorphic to the dicyclic group of order 12; $C$ is isomorphic to the dihedral group of order 12; and $D\cong (\Z/4\Z)\times(\Z/2\Z)$. The group $A$ has two maximal subgroups up to conjugacy in $G$, namely, a subgroup $H$ of order 6, and a subgroup $M$ of order 4. The group $D$ has a unique non-cyclic maximal subgroup of order 4, denoted $K$, and two maximal subgroups that are cyclic of order 4; one of them is $M$, and the other will be denoted $I$. 
The groups $M,I$ are conjugate in $S_{12}$, so $M\equiv I$. The following relations will be used throughout the proofs in this section:
\[I\le B,\; I\le D,\; M\le A,\;M\le D,\; H\le C,\; K\le C,\;M\equiv I.\]

\begin{lem}\label{G4curves_auto} With notation as above, the following hold.
\begin{enumerate}[(a)]
\item\label{G4crva} $\pi_A(\Q)=\{4\}\cup\Im\alpha$, where $\alpha(t)=4t^2/(t^2+3)$.
\item\label{G4crvb} $\pi_B(\Q)=\{4\}\cup\Im\beta$, where $\beta(t)=4t^2/(t^2+15)$.
\item\label{G4crvc} $\pi_C(\Q)=\{0,4\}$.
\item\label{G4crvcd} $\pi_D(\Q)=\Im\delta$, where $\delta(t)=(t-1)^2(t+2)$.
\item\label{G4crvci} $\pi_I(\Q)\supseteq\Im\iota$, where $\iota(t)=\beta[45(1 + t - t^2)/((2t - 1)(t^2 - t - 11))]$.
\item\label{G4crvcm} $\pi_M(\Q)\supseteq\Im\mu$, where $\mu(t)=\alpha[9t(t + 1)/((t - 1)(2 + t)(2t + 1))]$.
\end{enumerate}
\end{lem}
\begin{proof} Defining polynomials for the fixed fields of $A,B,C$ are given by
\begin{align*}
q_A &=x^2- g_A(t)\cdot f_A(t)^2,\\
q_B &= x^2-g_B(t)\cdot f_B(t)^2,\\
q_C &=x^2 + (15t^7 - 30t^6)x + 45t^{14} - 135t^{13} + 45t^{12},
\end{align*}

where
\begin{align*}
g_A(t)&=3t(4 - t),\\
g_B(t)&=15t(4 - t),\\
f_A(t) &=3t^{16}(t - 4)(t^2 - 3t + 1)(t^3 - 2t^2 + 18t - 54),\\
f_B(t) &=9t^{22}(t - 4)^2(t^2 - 3t + 1)(t^3 - 2t^2 + 18t - 54).\\
\end{align*}

\vspace{-5mm}
From the expression given for the polynomial $q_A$, it is clear that $Y_A$ is birational to the affine curve $y^2=g_A(t)$. Parametrizing the latter curve we obtain a birational map $\A^1\ratmap Y_A$ given by $v\mapsto (\alpha(v), z(v))$ for some rational function $z$. This map is defined everywhere on $\A^1(\Q)$, and its inverse is defined at every rational point different from $(0,0)$ and $(4,0)$. Now \eqref{G4crva} follows easily, and \eqref{G4crvb} is proved similarly. 

To prove \eqref{G4crvc} we note that the polynomial $q_C$ is quadratic in $x$, so $Y_C$ is isomorphic to the curve $y^2=\disc q_C=45(t - 4)^2t^{12}$,
whose only rational points are $(t,y)=(0,0)$ and $(4,0)$. Hence, $Y_C$ has two rational points, whose $t$-coordinates must be 0 and 4. Thus $\pi_C(\Q)=\{0,4\}$.

The curve $Y_D$ has genus 0. Using the nonsingular point $(2,256)\in Y_D(\Q)$ we compute a birational map $\A^1\ratmap Y_D$ whose first component is the rational function $\delta$. This map is defined everywhere on $\A^1(\Q)$, and its inverse is defined at every rational point different from $(0,0)$ and $(4,-40960)$. Part \eqref{G4crvcd} now follows by observing that $0,4\in\Im\delta$. 

To prove \eqref{G4crvci} we parametrize the curve $Y_I$ by an indirect approach using the fact that $I\le B$ and $I\le D$, and using the known parametrizations of $Y_B$ and $Y_D$. The equation $\beta(x)=\delta(y)$ defines a rational curve $Y$, and there is a birational map $\A^1\ratmap Y$ whose first component is 
\[i(v)=\frac{45(1 + v - v^2)}{(2v - 1)(v^2 - v - 11)}.\]
Defining $\iota=\beta\circ i$, we find that the polynomial $q_I(\iota(v),x)\in\Q(v)[x]$ has two roots in $\Q(v)$; choosing a root $p$, we obtain the desired parametrization $\A^1\ratmap Y_I$ given by $v \mapsto (\iota(v),p(v))$. Finally, we check that this map is defined everywhere on $\A^1(\Q)$, from which \eqref{G4crvci} follows. The proof of \eqref{G4crvcm} is very similar, using the parametrizations of $Y_A$ and $Y_D$, and the fact that $M\le A$ and $M\le D$.
\end{proof}

\begin{prop}\label{C2G4_autos}
The following statements hold for $v\in\Q\setminus\{0,3/2,4\}$.
\begin{enumerate}[(a)]
\item\label{C2G4autosi} If $v\in\Im\iota$, then $G_{4,v}\equiv I$.
\item\label{C2G4autosm} If $v\in\Im\mu$, then $G_{4,v}\equiv M\equiv I$.
\item\label{C2G4autosabd}If $v\notin\Im\alpha\cup\Im\beta\cup\Im\delta$, then $G_{4,v}\equiv V$.
\item\label{C2G4autosam} If $v\in\Im\alpha$, then $G_{4,v}\equiv A$ or $I$.
\item\label{C2G4autosbi} If $v\in\Im\beta$, then $G_{4,v}\equiv B$ or $I$.
\item\label{C2G4autosdjm} If $v\in\Im\delta$, then $G_{4,v}\equiv D$ or $I$.
\end{enumerate}
\end{prop}
\begin{proof} For the polynomials $\Phi_4,q_A,q_B,\ldots, q_M$, the only rational roots of the discriminants are the excluded values $0,3/2,4$. Let $P$ be a prime of the splitting field of $\Phi_4$ lying over the prime $(t-v)\subseteq\Q(t)$. By Lemma \ref{G4curves_auto} we have $D_P\not\le C$; moreover, since $H\le C$ and $K\le C$, then $D_P\not\le H$ and $D_P\not\le K$.

If $v\in\Im\iota$, then we may assume $D_P\subseteq I$. The group $I$ has a unique maximal ideal $J$, and $J\le C$. It follows that $D_P=I$, proving \eqref{C2G4autosi}. A similar argument proves \eqref{C2G4autosm}, as $J$ is also the unique maximal ideal of $M$. Part \eqref{C2G4autosabd} follows from the fact that $A,B,C,D$ are all the maximal subgroups of $G$.

If $v\in\Im\alpha$, we may assume $D_P\subseteq A$. The maximal subgroups of $A$ are $H$ and $M$, and we have noted that $D_P$ cannot be a proper subgroup of $M$. Thus, if $D_P\ne A$ and $D_P\ne M$, then $D_P\subseteq H$, a contradiction. Hence $D_P=A$ or $M$; in the latter case, $D_P\equiv I$. This proves \eqref{C2G4autosam}. Similar reasoning proves \eqref{C2G4autosbi} and \eqref{C2G4autosdjm}, using the fact that the maximal subgroups of $B$ are $H,I$, and the maximal subgroups of $D$ are $M,K,I$.
\end{proof}

\subsection{Maps with trivial automorphism group}\label{trivial_section_G4} 
For maps of the form $\phi_v$ as in \eqref{no_auto_maps}, we study here the dynatomic groups $G_{4,v}:=G_{4,\phi_v}$. Unlike in previous sections, a complete classification of these groups is precluded by the appearance of several curves of high genus whose rational point sets we are unable to determine. Our goal is instead to construct a list of groups, as short as possible, that is known to contain all groups of the form $G_{4,v}$. Our main result in this direction is Proposition \ref{C2G4}. The technique of Galois group specialization is once again the main tool used, with one new aspect being the use of the method of Chabauty and Coleman to determine rational points on curves. A combination of this method with a Mordell--Weil sieve is implemented in \textsc{Magma}, and applies to curves of genus 2 whose Jacobian has Mordell--Weil rank 1; see \cite{bruin-stoll}*{\textsection 4.4} for details.

Let $W=(\Z/4\Z)\wr S_3$. With $\phi$ as in \eqref{genphi}, let $\Phi_4=\Phi_{4,\phi}$ and let $G$ be the Galois group of $\Phi_4$. We compute that $G\equiv W$. Up to conjugacy, $G$ has 164 subgroups, all of which could \emph{a priori} occur as a dynatomic group $G_{4,v}$. Our analysis in this section will narrow down the possibilities to 26 groups, 10 of which are known to be realized in the form $G_{4,v}$. Despite the incompleteness of this classification, the data provided by the above list of 26 groups is sufficient for deriving our main applications in Section \ref{applications_section}.

We begin by labeling the relevant subgroups of $G$. When convenient, we use the identification scheme of the \textsc{Small groups} library \cite{smallgps}, in which groups are identified by a pair $(o,n)$, where $o$ is the order of the group.

The group $G$ has five maximal subgroups up to conjugacy, labeled $M_i$ for $1\le i\le 5$; these are uniquely determined by the following properties:
\[M_1\cong(192,944); |M_2|=128; |M_3|=96; M_4\cong (192,188); M_5\cong(192,182).\]

The group $M_1$ has five maximal subgroups up to conjugacy in $G$; these are labeled $A_i$ according to the following properties:
\[|A_1|=48;\;|A_2|=64;\;A_3\cong(96,68).\]
The groups $A_4$ and $A_5$ are both isomorphic to $(96,64)$, but $A_4$ contains an element whose disjoint cycle decomposition is a product of six 2-cycles, and $A_5$ does not contain such an element.

The group $M_2$ has seven maximal subgroups labeled as follows:

\begin{itemize}
\item $B_2\cong(64,198)$, $B_3\cong(64,85)$, $B_5\cong(64,55)$;
\item $B_1\cong(64,101)$ and does not have an element with cycle decomposition of the form $(8\text{-cycle})(4\text{-cycle})$;
\item $B_4\cong (64,20)$ and does not have an element with cycle decomposition of the form $(8\text{-cycle})(4\text{-cycle})$;
\item $B_6\cong(64,20)$ and has an element with cycle decomposition of the form $(8\text{-cycle})(4\text{-cycle})$;
\item $B_7\cong(64,101)$ and has an element with cycle decomposition of the form $(8\text{-cycle})(4\text{-cycle})$.
\end{itemize}

We label two maximal subgroups of $M_3$, uniquely determined by the conditions $|C_1|=24$ and $|C_2|=32$. Finally, we label one maximal subgroup $J$ of $B_1$, and one maximal subgroup $K$ of $B_2$, as follows:

\begin{itemize}
\item $J\cong (32,11)$, does not have an element with cycle decomposition of the form $(8\text{-cycle})(2\text{-cycle})(2\text{-cycle})$, and does not have an element whose disjoint cycle decomposition is a product of six 2-cycles.
\item $K\cong (32, 25)$, contains a 4-cycle and an element with cycle decomposition of the form $(4\text{-cycle})(2\text{-cycle})(2\text{-cycle})$.
\end{itemize}

For each subgroup $H$ of $G$ that was labeled above, we compute a defining polynomial $q_H$ for the fixed field of $H$, and we find the rational roots of $\disc q_H$. The rational numbers thus obtained are all included in the set
\[\Delta:=\{0,1,3/2,8/3,11/2\}.\]

The following collections of subgroups of $G$ will be central to our analysis:
\begin{align*}
\calF&:=\{M_4, M_5, A_2 ,A_3, A_4, A_5, B_3, B_4, B_6, B_7,J\},\\
\calR&:=\{G,M_1,M_2,M_3,A_1,B_1,B_2,C_1,C_2,K\}.
\end{align*}

\begin{lem}\label{C2G4_almost_all}
For every $H\in\calR$ there exists $v\in\Q\setminus\{0\}$ such that $G_{4,v}\equiv H$.
\end{lem}
\begin{proof}
For $H=G$, the result follows from the Hilbert irreducibility theorem. For the remaining groups $H\in\calR$, we carry out a search for rational points of small height on the curve $Y_H$ in order to find elements $v\in\pi_H(\Q)$; we then test these values of $v$ to check the condition $G_{4,v}\equiv H$. In this way we obtain the following pairs $(v,H)$:
\begin{align*}
&(11/2,M_1),\;(1,M_2),\;(3/2,M_3),\;(25/6,A_1),\\
 &(1/6,B_1),\;(27/7,B_2),\;(5,C_1),\;(125/21,C_2),\;(8/3,K).\qedhere
 \end{align*}
\end{proof}

Lemma \ref{C2G4_almost_all} shows that groups in $\calR$ are realizable as dynatomic groups $G_{4,v}$. The next lemma implies that groups in $\calF$ cannot be realized as such.

\begin{lem}\label{G4curves}
For every group $H\in\calF$ we have $\pi_H(\Q)\subseteq\{0,1\}$. Moreover, $\pi_K(\Q)=\{0,8/3\}$.
\end{lem}
\begin{proof}
For $H\in\calF$ we determine the set of rational points $Y_H(\Q)$, and in particular the set $\pi_H(\Q)$, using the following observations.

For $H\in\{M_4,A_5,J\}$, the curve $Y_H$ has genus 0. We compute a birational map $Y_H\ratmap C$, where $C$ is a conic. The existence of a rational point on $C$ can be tested by local means; in all three cases, we find that $C$ has no rational point. Thus, the only rational points on $Y_H$ are those where the computed rational map to $C$ is undefined.
 
For $H\in\{M_5,B_6\}$, $Y_H$ has genus 1. A search for nonsingular points of small degree yields a point defined over the quadratic field $L=\Q(\sqrt {-1})$. We may thus compute a birational map $Y_H\ratmap E$ defined over $L$, where $E$ is an elliptic curve. We find that $E$ has rank 0 over $L$; computing the set $E(L)$ and pulling back the points in $E(L)$, we determine the set $Y_H(L)$ and therefore $Y_H(\Q)$.
 
 For $H\in\{A_2,A_4,B_7,K\}$, $Y_H$ is birational to a curve of genus 2 and Mordell--Weil rank 0 or 1. The \textsc{Magma} functions \texttt{Chabauty0} and \texttt{Chabauty} successfully compute the set $Y_H(\Q)$.
 
For $H\in\{B_3,B_4\}$, $Y_H$ is birational to a hyperelliptic curve $C$ of genus 2 or 3. Using the methods of \cite{bruin}*{\textsection 5.5}, implemented in \textsc{Magma} via the function \texttt{HasPointsEverywhereLocally}, we show that $C$ no rational point. Finally, for $H=A_3$, $Y_H$ admits a rational map to a hyperelliptic curve of genus 1 with no rational point, again proved by a local argument.
\end{proof}

\begin{lem}\label{G4ab_lem} Up to conjugacy, there are $16$ subgroups of $G$ that do not belong to $\calR$ and are not contained in any group in $\calF\cup\{K\}$.
\end{lem}
\begin{proof}
This can be proved by computing the lattice of conjugacy classes of subgroups of $G$ and traversing the lattice to verify the given statement.
\end{proof}

Let $\calU$ denote the set of 16 groups referred to in Lemma \ref{G4ab_lem}.

\begin{prop}\label{C2G4}
For every $v\in\Q\setminus\{0\}$ we have $G_{4,v}\equiv H$ for some group $H$ in the set $\calP:=\calR\cup\calU$.
\end{prop}
\begin{proof}
 Computing $G_{4,v}$ for $v\in\Delta\setminus\{0\}$, we obtain the groups $M_1,M_2,K\in\calR$. Now let $v\in\Q\setminus\Delta$, and let $P$ be a prime of the splitting field of $\Phi_4$ lying over the prime $(t-v)$ of $\Q(t)$, so that $G_{4,v}\equiv D_P$. Since $v\notin\Delta$, Lemma \ref{G4curves} implies that $D_P$ is not contained in any group in $\calF\cup\{K\}$. Thus, by Lemma \ref{G4ab_lem}, $D_P$ must be conjugate to some group in $\calR\cup\calU$.
\end{proof}

\begin{rem}
The groups in the set $\calU$ are those whose realizability in the form $G_{4,v}$ is unknown; in particular, a search for rational points on the curves $Y_H$ for $H\in\calU$ failed to find $v$ such that $G_{4,v}\equiv H$. Moreover, we were unable to solve the problem of determining all rational points on $Y_H$ for $H\in\calU$ by standard means. As an example, two of the groups in $\calU$ are contained in the group $C_1$, so their corresponding curves $Y_H$ have dominant rational maps onto $Y_{C_1}$. We are thus led to the problem of finding all rational points on $Y_{C_1}$, a geometrically irreducible curve of genus 8 for which we  do not know basic information such as its automorphism group or whether it is hyperelliptic, due to the excessive computing time required. Nevertheless, a computation in \textsc{Magma} shows that every curve $Y_H$ with $H\in\calU$ has genus greater than 1, so that $H$ can occur as $G_{4,v}$ for at most finitely many parameters $v$.
\end{rem}

\section{Applications}\label{applications_section}
The Galois-theoretic information obtained in Sections \ref{Phi3_section} and \ref{Phi4_section} can be used to derive several types of results concerning the arithmetic dynamics of quadratic maps with a $2$-periodic critical point. In this section we prove several results for maps with trivial automorphism group; the interested reader can prove similar results for maps with nontrivial automorphisms.

\subsection{Degrees of periodic points}\label{per_degree_section} By the \textbf{degree} of a point $P\in\PP^1(\overline\Q)$ we mean the degree of the field extension $\Q(P)/\Q$. The following proposition proves Theorem \ref{intro_deg_dens_cor}\eqref{intro_deg_thm} and an analogous statement for $4$-periodic points. 

\begin{prop}\label{G34period_deg}
Suppose that $\phi\in\Rat_2(\Q)$ has a $2$-periodic critical point, and that $\Aut(\phi)$ is trivial. Then the following hold.
\begin{enumerate}[(a)]
\item\label{period_deg3} Every $3$-periodic point of $\phi$ has degree $6$.
\item\label{period_deg4} Every $4$-periodic point of $\phi$ has degree $2,4,6,8$, or $12$.
\end{enumerate}
\end{prop}
\begin{proof}
The map $\phi$ is linearly conjugate over $\Q$ to a map $\phi_v$ as in \eqref{no_auto_maps}, so it suffices to prove the result for $\phi_v$. The point $\infty$ is $2$-periodic for $\phi_v$, so the set of all $3$-periodic points of $\phi_v$ is the set of roots of $\Phi_{3,v}$, a polynomial of degree 6. Moreover, $\Phi_{3,v}$ is irreducible: by Proposition \ref{C2G3}, we have $G_{3,v}\equiv H$ with $H\in\{W, C,G_{3,3}\}$, and each of these three groups is a transitive subgroup of $S_6$. This proves \eqref{period_deg3}.

To prove \eqref{period_deg4}, suppose that $H\subseteq S_{12}$ is a permutation representation of the Galois group $G_{4,v}$. Since the $4$-periodic points of $\phi_v$ are the roots of $\Phi_{4,v}$, basic Galois theory implies that the degrees of these points are equal to the indices $[H:H_i]$, where $H_i$ is the stabilizer of $i\in\{1,2,\ldots, 12\}$ in $H$.

By Proposition \ref{C2G4}, we have $G_{4,v}\equiv H$ for some $H\in \calP$. It follows that the degree of any 4-periodic point of $\phi_v$ belongs to the set
\[\{[H:H_i]\mid H\in\calP, 1\le i\le 12\}.\]
Computing the above set of indices, we obtain $\{2,4,6,8,12\}$.
\end{proof}

\subsection{Density results}\label{density_section}
For every nonconstant polynomial $F\in\Q[x]$, let $S_F$ denote the set of prime integers $p$ such that $F$ has a root in the $p$-adic field $\Q_p$. The Dirichlet density of the set $S_F$, denoted here by $\delta(S_F)$, can be computed from data associated to the Galois group of $F$; more precisely, we have the following formula, which is easily deduced from \cite{krumm_lgp}*{Theorem 2.1}.

\begin{lem}\label{density_lem}
Let $F\in \Q[x]$ be a separable polynomial of degree $n\ge 1$. Let $L$ be a splitting field of $F$, and set $G=\Gal(L/\Q)$. Let $\alpha_1,\ldots, \alpha_n$ be the roots of $F$ in $L$. For each index $i$, let $G_i$ denote the stabilizer of $\alpha_i$ under the action of $G$. Then the density of $S_F$ is given by
\[\delta(S_F)=\frac{\left|\bigcup_{i=1}^n G_i\right|}{|G|}.\]
\end{lem}

The following proposition proves Theorem \ref{intro_deg_dens_cor}\eqref{intro_density_thm} and an analogous statement for $4$-periodic points.

\begin{prop}\label{G34_density_prop} Suppose that $\phi\in\Rat_2(\Q)$ has a $2$-periodic critical point, and that $\Aut(\phi)$ is trivial. Then the following hold.
\begin{enumerate}[(a)]
\item\label{G3_density} The density of the set of primes $p$ such that $\phi$ does not have a $3$-periodic point in $\PP^1(\Q_p)$ is at least $13/18$.
\item\label{G4_density} The density of the set of primes $p$ such that $\phi$ does not have a $4$-periodic point in $\PP^1(\Q_p)$ is at least $3/8$.
\end{enumerate}
\end{prop}
\begin{proof}
The map $\phi$ is linearly conjugate over $\Q$ to $\phi_v$ for some $v\in\Q\setminus\{0\}$. Thus, $\phi$ has a $3$-periodic point in $\PP^1(\Q_p)$ if and only if $\Phi_{3,v}$ has a root in $\Q_p$. As seen in the proof of Proposition \ref{G34period_deg}, the Galois group $G_{3,v}$ is isomorphic to one of three known groups.  Applying Lemma \ref{density_lem} with $F=\Phi_{3,v}$ and $G$ equal to each of the three possibilities for $G_{3,v}$, we obtain two possible values for $\delta(S_F)$, namely $1/6$ and $5/18$. The complement of $S_F$ therefore has density at least $13/18$; this proves \eqref{G3_density}. Part \eqref{G4_density} is proved similarly, using the groups in the set $\calP$ from Proposition \ref{C2G4} as the possibilities for $G_{4,v}$. In this case there are 16 possible values of $\delta(S_F)$, the largest one being $5/8$. This completes the proof. If we use only the groups that are known to be realized as $G_{4,v}$, namely those from Proposition \ref{C2G4_almost_all}, we obtain 10 possible densities, the largest being $39/64$.
\end{proof}

\subsection*{Acknowledgements} We thank Nicole Sutherland for helpful conversations on computing absolute factorizations and Galois groups, and we thank the anonymous referee for a careful reading of our paper and for substantial comments on the presentation of our code.

%\subsection*{Data availability} Data sharing not applicable to this article as no datasets were generated or analyzed during the current study.

%\subsection*{Conflict of interest} The authors have no conflict of interest.

\end{document}